\documentclass[12pt,reqno]{amsart}

\usepackage{amsfonts,amssymb,amsmath,amsthm}
\usepackage{enumerate}
\usepackage{graphicx}
\usepackage{url}
\usepackage{verbatim}
\usepackage{xcolor}

\allowdisplaybreaks[4]

\newtheorem{theorem}{Theorem}[section]
\newtheorem{lemma}[theorem]{Lemma}

\newtheorem{conjecture}[theorem]{Conjecture}
\newtheorem{corollary}[theorem]{Corollary}
\theoremstyle{definition}

\numberwithin{equation}{section}
\newcommand{\locref}[1]{\ifthenelse{\pageref{#1}<\thepage}{above}{below}}

\author{Redmond, Timothy 
  \and
  Ryavec, Charles
  }


\begin{document}
\bibliographystyle{plain}

\title[Bessel K Series for the Riemann Zeta Function]{Bessel K Series for the Riemann Zeta Function}
\begin{abstract}
This paper provides some expansions of the Riemann xi function, $\xi$,
as a series of Bessel K functions.
\end{abstract}
\maketitle

\section{Introduction}

Some expansions of the Riemann $\xi$ function (and expansions of related functions) are provided in this article as sums of the form, 
\begin{equation*}
\sum_{j=1}^{\infty} c_j(s) K_s(x_j),
\end{equation*}
 where $K_s(x)$ denotes the Bessel K function. Theorem~\ref{EasyBessel} and Theorem~\ref{ris4} use transformation formulas of theta functions. Conjecture~\ref{ZetaDecomposition} is based on a partition of unity different from the one Riemann used in his memoir. Partitions of unity, such as that employed by Riemann, are more or less equivalent in the derivation of the meromorphy and functional equation of $\zeta(s)$, but we found only one of them that gave Bessel expansions. We suspect that among algebraic functions of $\lambda$ there are others. One of these is given in the appendix.  

 Half of the appendix is a discussion of a curious zeros phenomenon which occured in relation to partitions of unity that arose in the use of Hecke operators. Some numerical data  hinted that an eigenform for Hecke operators would be interesting. We were able to construct a formal eigenform for weight one half Hecke operators indexed over odd indices, but we could not show that it converged.   

Conjecture~\ref{ZetaDecomposition} is a provisional result, based on a conjectured upper bound on certain partition coefficients. The result holds up under numerical tests, and it looks like an elementary proof, somewhat lengthy, would establish the upper bounds in the same way as the elementary proofs establish the upper bounds of the standard partition function.

 As usual,
\begin{align*}
\theta_2 &=& \sum_{n=-\infty}^{\infty}q^{(n - 1/2)^2}\\
\theta_3 &=& \sum_{n=-\infty}^{\infty}q^{n^2}\\
\theta_4 &=& \sum_{n=-\infty}^{\infty}(-1)^n  q^{n^2}\\
\lambda &=& \left(\frac{\theta_2}{\theta_3}\right)^4.
\end{align*}

\section{A brief note}\label{BriefNote1Section}

For the purposes of much of this paper, it will be convenient to define the
functions, $\theta_2$, $\theta_3$, $\theta_4$ and $\lambda$ as functions of $y$
where $y$ is real positive and $q=e^{-\pi y}$.

\section{A first $\zeta(s)$ Bessel Series}

\begin{theorem}\label{EasyBessel}
  Define a multiplicative function, $a(n)$, on prime powers by, 
\begin{align*}
a(n) &= \sigma_1(n)    & n  & = p^f, \ p>2 \\
    &= n            & n  & = 2^f.
\end{align*}
This is to say that the $a(n)$ are the coefficients in the expansion,
$$
\dot\theta_4/\theta_4 = 2 \pi \sum_{n=1}^{\infty}a(n) q^n, 
$$
where the derivative is taken with respect to $y$.

 Then,
\begin{align*}
  \frac{s}{4\pi}&\left(2^{\frac{s-1}{2}} - 2^{\frac{1-s}{2}}\right)\pi^{-\frac{s}{2}}\Gamma\left(\frac{s}{2}\right)\zeta(s) \\
  & = \sum_{n=1}^{\infty}a(n) n^{-\frac{1+s}{4}}\sum_{m=0}^{\infty} (2m+1)^{\frac{s+1}{2}} K_{\frac{s+1}{2}}(\pi \sqrt{n}(2m+1)).
\end{align*}
\end{theorem}

\begin{proof}
For $\sigma > 0$,   
\begin{align*}
  (1 - 2^{1-s})&\pi^{-\frac{s}{2}}\Gamma\left(\frac{s}{2}\right)\zeta(s)  \\
  &= -\frac{1}{2}\int_0^{\infty} y^{\frac{s}{2}-1}(\theta_4 - 1)dy \\
&=\frac{1}{s}\int_0^{\infty} y^ {\frac{s}{2}}\frac{\dot \theta_4}{\theta_4} \theta_4 dy \\
&= \frac{1}{s}\int_0^{\infty} y^ {\frac{s}{2}}\frac{\dot \theta_4}{\theta_4}(y) \frac{1}{\sqrt{y}}\theta_2\left(\frac{1}{y}\right) dy \\
&= \frac{1}{s}\int_0^{\infty} y^ {\frac{1+s}{2}-1}\frac{\dot \theta_4}{\theta_4}(y) \  \theta_2\left(\frac{1}{y}\right) dy \\
&= \frac{1}{s}\sum_{n=1}^{\infty}(2\pi )a_n \sum_{m=0}^{\infty}2 \int_0^{\infty} y^ {\frac{1+s}{2}-1}e^{-\pi(n y + \frac{(2 m + 1)^2}{4 y})} dy \\
&= \frac{4\pi}{s} \sum_{n=1}^{\infty}a_n\sum_{m=0}^{\infty} \Big(\frac{2m+1}{2\sqrt{n}}\Big)^{\frac{s+1}{2}}\int_0^{\infty} x^ {\frac{1+s}{2}-1} e^{\pi\frac{\sqrt{n}(2m+1)}{2}(x + \frac{1}{x})}dx  \\
&= \frac{4\pi}{s} \sum_{n=1}^{\infty}a_n\sum_{m=0}^{\infty} \Big(\frac{2m+1}{2\sqrt{n}}\Big)^{\frac{s+1}{2}} 2 K_{\frac{s+1}{2}}(\pi \sqrt{n}(2m+1)) \\
&= \frac{8\pi}{s} \sum_{n=1}^{\infty}a_n\sum_{m=0}^{\infty} \Big(\frac{2m+1}{2\sqrt{n}}\Big)^{\frac{s+1}{2}}  K_{\frac{s+1}{2}}(\pi \sqrt{n}(2m+1)),  
\end{align*}
 and the result follows on multiplying both sides by
\begin{equation*}
\frac{s}{8\pi}   \ 2^{\frac{s+1}{2}}.
\end{equation*}
\end{proof}

 The result provides a rapidly convergent series for $\zeta(s)$. How
 it would compare with the methods currently in use calculating zeros
 we cannot say. We include this result, and some others like it, as a comparison to the kinds of
 expansions using a partition of unity mentioned in the introduction.

\section{A second $\zeta(s)$ Bessel Series}

In preparation for Theorem 4.1,  define a second multiplicative functions, $b(m)$, on prime powers, by
\begin{eqnarray*}
b(m) &=& \sigma_1(m) \qquad m = p^f, \ p > 2 \\
      &=& 0 \qquad \qquad m  = 2^f.
\end{eqnarray*}
The $b(m)$ are therefore the coefficients in the expansion,
$$
\theta_2^4 = 16 \sum_{m=1}^{\infty} b(m) q^m. 
$$
With the $a(n)$ and $b(m)$ as given above, define entire functions of the complex variable, $s$, and positive integers, $j$, by 
\begin{equation*}
  c_j(s) = \sum_{d|j} a(d) b\left(\frac{j}{d}\right)\Big(\frac{j}{d^2}\Big)^{s/2},
\end{equation*}
which functions,
$$
c_j(s) = \prod_{p^f||j}c_{p^f}(s),
$$
are multiplicative in $j$ for all $s$. We then have,

\begin{theorem}\label{ris4}
\begin{equation*}
 \frac{s(s+1)}{32\pi^2\sqrt{2}}\Big(2^{\frac{s}{2}}-2^{-\frac{s}{2}}\Big)\Big(2^{\frac{s-1}{2}}-2^{-\frac{s-1}{2}}\Big)\zeta^{*}(s)\zeta^{*}(s+1) =  \sum_{j=1}^{\infty} c_j(s)K_s( 2 \pi \sqrt{j})
\end{equation*}
\end{theorem}
\begin{proof}
  The proof of theorem~\ref{ris4} is very much like the proof of theorem~\ref{EasyBessel}. Replace $\theta_4$ with $\theta_4^4$ on the right hand
  side. Use the formula for the number of representations of an
  integer as the sum of four squares on the left hand side.
\end{proof}

The asterisks indicate the usual,
\begin{eqnarray*}
\zeta^{*}(s) &=& \pi^{-\frac{s}{2}}\Gamma(\frac{s}{2})\zeta(s) \\ 
\zeta^{*}(s+1) &=& \pi^{-\frac{s+1}{2}}\Gamma(\frac{s+1}{2})\zeta(s+1).
\end{eqnarray*}
We mention that another BesselK expansion of zeta originates from $\theta_4^8$ in the same way as that obtained from $\theta_4^4$. There is also an expansion using $\theta_4^2$. 

Regarding theorem~\ref{ris4}, there is an additional result that we
can prove about  the zeros of $c_j(s)$.

\begin{theorem}\label{ris4czeros}
For any $j$,  all the zeros of $c_j(s)$ are pure imaginary.
\end{theorem}
\begin{proof}
 Since $c_j(s)$ is multiplicative, we only need to look at the zeros
 of $c_{p^k}(s)$ where $p$ is prime.  If $p=2$ there are no zeros of
 \begin{equation*}
   c_{2^m}(s) = 2^m2^{-m s/2}.
 \end{equation*}
 For $p$ an odd prime, the equation for $c_{p^k}(s)$ is as follows
 \begin{equation*}
   c_{p^k}(s) = \sum_{j=0}^{k}\dfrac{p^{j+1}-1}{p-1}\dfrac{p^{k-j+1}-1}{p-1}p^{(k-2j)s/2}.
 \end{equation*}
 Now we will perform a couple of changes of variable in turn:
 \begin{align*}
   z & = p^{s/2} \\
   e^{i\phi} & = z
 \end{align*}
 Note that $s$ is pure imaginary  exactly when $z$ is on the
 unit circle which  also corresponds to real $\phi$ between 0 and $2\pi$.
 We will use the following approximation of $c_{p^k}(s)$,
 \begin{equation*}
   c_{p^k}(s) \approx \dfrac{p^{k+2}}{(p-1)^2}\sum_{j=0}^kz^{(k-2j)}
   =\dfrac{p^{k+2}}{(p-1)^2}\dfrac{z^{k+1}-z^{-k-1}}{z-z^{-1}}.
 \end{equation*}
Note that the formula on the right has all its zeros when $z$ is on the unit
circle. We will attempt to use this approximation to show that the
zeros of $c_{p^k}(s)$ only occur when $z$ is on the unit circle. To make this
argument we use the following estimate for all $z$ on the unit circle:
\begin{align*}
  \Big|\dfrac{(p-1)^2}{p^{k+2}}&\dfrac{z-z^{-1}}{2i}c_{p^k}(s)-\dfrac{z^{k+1}-z^{-k-1}}{2i}\Big|
  \\
  & = \left|\dfrac{z-z^{-1}}{2i}\sum_{m=0}^k\left(\dfrac{1}{p^{k+2}}-\dfrac{1}{p^{m+1}}-\dfrac{1}{p^{k-m+1}}\right)z^{k-2m}\right| \\
  & < \sum_{m=0}^k\left(\dfrac{1}{p^{k+2}}-\dfrac{1}{p^{m+1}}-\dfrac{1}{p^{k-m+1}}\right) < 1
\end{align*}
If we now make the change of variable
\begin{equation*}
  z = e^{i\phi}
\end{equation*}
then this inequality indicates that
\begin{equation*}
  \dfrac{(p-1)^2}{p^{k+2}}\sin{\phi}\,c_{p^k}(s)
\end{equation*}
differs from
\begin{equation*}
  \sin{(k+1)\phi}
\end{equation*}
by less than one. This means that for each integer $m$ in the interval
\begin{equation*}
  0 < m < 2(k+1)
\end{equation*}
where $m\neq k+1$ we have a zero for $c_{p^k}(s)$ for some $\phi$ in the range
\begin{equation*}
  \dfrac{1}{k+1}\left(m-\dfrac{1}{2}\right)\pi<\phi<\dfrac{1}{k+1}\left(m+\dfrac{1}{2}\right)\pi.
\end{equation*}
Since $c_{p^k}(s)$ is $z^{-k}$ times a degree $2k$ polynomial in $z$
this accounts for all the zeros of $c_{p^k}(s)$.
\end{proof}

\section{A Partition of Unity}

 The function,
\begin{equation}
 \zeta^{*}(s,r) = \frac{1}{2r} \int_0^{\infty} y^{\frac{r s}{2}-1}\left(\theta^r -1- \frac{1}{y^{\frac{r}{2}}}\right) dy \tag{$0 < \sigma < 1$}
\end{equation}
$$
\theta = \theta_3,
$$
 may be combined with the partition of unity,
\begin{eqnarray*}
1 &=& (1 - \lambda( y)) + \lambda( y) \\
&=& \lambda\left(\frac{1}{y}\right) + \lambda( y),
\end{eqnarray*}
together with two integrations by parts, to give two meromorphic expressions in lines 3 and 4 below:
\begin{eqnarray*}
2r \zeta^{*}(s,r) &=&  \int_0^{\infty} y^{\frac{r s}{2}-1}(\theta^r (1 - \lambda) -1) dy + \int_0^{\infty} y^{\frac{r s}{2}-1}(\theta^r \lambda - \frac{1}{y^{\frac{r}{2}}}) dy \\
&=&  \int_0^{\infty} y^{\frac{r s}{2}-1}(F_r(y) -1) dy + \int_0^{\infty} y^{\frac{r(1-s)}{2}-1}(F_r(y)-1) dy \\
&=&  -\frac{2}{rs} \int_0^{\infty} y^{\frac{r s}{2}}\dot F_r dy  - \frac{2}{r(1-s)} \int_0^{\infty} y^{\frac{r(1-s)}{2}}\dot F_r dy  \\
 &=& \frac{4}{r^2 s(s-1)} \int_0^{\infty}\Big(y^{-\frac{r(1-s)}{2}} + y^{-\frac{rs}{2}}\Big)(\dot F_r y^{\frac{r}{2} + 1})'dy.
\end{eqnarray*}
In the above, the notation is,
$$
F_r = \theta^r (1 - \lambda).
$$
We have chosen not to work with the entire function that comes from the fourth line above; i.e.,    
\begin{align*}
Z(s,r)  &=  \frac{r^3}{2} s (s-1) \zeta^{*}(s,r) \\
&= \int_0^{\infty}\Big(y^{-\frac{r(1-s)}{2}} + y^{-\frac{rs}{2}}\Big)(\dot F_r y^{\frac{r}{2} + 1})'dy \\
&=   \begin{aligned}[t]
   \int_0^{\infty}&\Big(y^{\frac{rs}{2}+1} + y^{\frac{r(1-s)}{2}+1}\Big)\ddot F_r dy \\
&+ \left(\frac{r}{2}+1\right)\int_0^{\infty}\Big(y^{\frac{rs}{2}} +
   y^{\frac{r(1-s)}{2}}\Big)\dot F_r dy,
   \end{aligned}
\end{align*}
 but rather with the expression that precedes it; i. e., with
$$
r^2 \zeta^{*}(s,r) = -\Big(\frac{\Psi_r(s)}{s} + \frac{\Psi_r(1-s)}{1-s}\Big),
$$
where 
$$
\Psi_r(s) = \int_0^{\infty} y^{\frac{r s}{2}}\dot F_r dy.
$$
The expression, $\Psi_r(s)$, with one integration by parts will
introduce one set of coefficients, $a_r(n)$, to consider in upcoming
formulas, while two integrations by parts involves two sets of coefficients,
$a_r(n), a_r'(n)$, and a bit more bookkeeping. We stick with the simpler single integration by parts, and present three expansions for $r = 1, 2, 3$. While two integrations by parts gives pure Bessel expansions in all three cases, in the absence of any insight into the value of Bessel expansions for the study of zeta functions, we go for the method with fewer coefficients. For example, if the prime number theorem could be established via Theorem 4.1, then Bessel functions would be in the game. As things are, they're a step too far or a step in the wrong direction. 
 
Define sequences, $a_r(n)$, for $r = 1, 2, 3$, by
\begin{align*}
4 \frac{\dot\theta_4}{\theta_4} - (4 - r) \frac{\dot\theta_3}{\theta_3}  &= 4(2\pi) \sum_{n=1}^{\infty} a(n)q^n -(4-r)(2\pi)\sum_{n=1}^{\infty} a(n)(-1)^n q^n   \\
 &= 2\pi \sum_{m=1}^{\infty}\Big(4 + (4 -r)(-1)^n \Big) a(n) q^n\\
 &= \sum_{m=1}^{\infty}a_r(n) q^n, \\
\end{align*}
where the $a(n)$ are defined earlier:
\begin{equation*}
  a(n) = \left\{\begin{array}{cl}
  n & \mbox{if $n$ is a power of 2} \\
  \sigma_1(n) & \mbox{otherwise}.
  \end{array}\right.
\end{equation*}
We define sequences, $b_r(n)$, for $r = 1, 2, 3$, by
\begin{align*}
\theta_3^r \lambda  &= \theta_3^{r-4}\theta_2^4 \\
 &=  \sum_{m=1}^{\infty}b_r(n)q^n.
\end{align*}

\section{A Provisional Result}

\begin{conjecture}\label{ZetaDecomposition}
For $r = 1, 2, 3,$
$$
 \zeta^{*}(s, r) = -\Big(\frac{\Psi_r(s)}{s} + \frac{\Psi_r(1-s)}{1-s}\Big),
$$ 
 where,  
\begin{eqnarray*}
\Psi_r(s) &=& 2 \pi \sum_{j=1}^{\infty} c_j(s, r)K_{\frac{rs+2 - r}{2}}( 2 \pi \sqrt{j})  \\
c_j(s, r) &=& \sum_{d|j} a_r(d) b_r\left(\frac{j}{d}\right)\left(\frac{j}{d^2}\right)^{\frac{r s+2-r}{4}} \qquad j\ge 1
\end{eqnarray*}
\end{conjecture}

\begin{proof}
    We have,  
\begin{eqnarray*}       
 \Psi_r(s) &=& \int_0^{\infty} y^{\frac{rs}{2}} \dot F_r (y) dy \\
 &=&  \int_0^{\infty} y^{\frac{rs}{2}} \Big( 4\frac{\dot\theta_4}{\theta_4} - (4 - r) \frac{\dot\theta_3}{\theta_3}  \Big)F_r (y) dy \\
&=& \int_0^{\infty} y^{\frac{rs}{2}} \Big( 4\frac{\dot\theta_4}{\theta_4} - (4 - r) \frac{\dot\theta_3}{\theta_3}  \Big)\frac{1}{y^{\frac{r}{2}}}\theta^r(\frac{1}{y})\lambda(\frac{1}{y}) dy \\
 &=&   \sum_{m, n = 1}^{\infty}a_r(n) b_r(m) \int_0^{\infty} y^{\frac{rs+2 -r}{2}-1} e^{-\pi (n y + \frac{m}{y})}  dy \\ 
 &=&   \sum_{m, n = 1}^{\infty}a_r(n) b_r(m)\Big(\frac{m}{n}\Big)^{\frac{rs+2 -r}{4}} \int_0^{\infty} y^{\frac{rs+2 -r}{2}-1} e^{-\pi \sqrt{mn}(y +\frac{1}{y})}  dy \\
 &=&   2 \sum_{m, n = 1}^{\infty}a_r(n) b_r(m)\Big(\frac{m}{n}\Big)^{\frac{rs+2 -r}{4}}  K_{\frac{rs+2 - r}{2}}( 2 \pi \sqrt{mn}),     
\end{eqnarray*}
where we recall,
$$
2 K_w(x) =  \int_0^{\infty}u^{w - 1}e^{-\frac{x}{2}(u + \frac{1}{u})} du, \qquad x > 0.
$$  
 The result follows from Dirichlet convolution, collecting $m$ and $n$ according to, $j = mn.$
\end{proof}
The result is provisional upon the justification of the interchange of
the sum and integral in the fourth step.

\section{Upper Bounds}

We conjecture that for any 
$\epsilon >0$, then 
\begin{align*}
b_1(m) &\le e^{\pi\sqrt{3m}(1+\epsilon)} \qquad m\ge m((\epsilon) \\
b_2(m) &\le e^{\pi\sqrt{2m}(1+\epsilon)} \qquad m\ge m((\epsilon) \\
b_3(m) &\le e^{\pi\sqrt{m}(1+\epsilon)} \qquad m\ge m((\epsilon).
\end{align*}
We have numerical evidence for this conjecture and it conforms with the
elementary estimates as obtained in~\cite{ErdosElementary}.

We recall the estimate, 
$$
K_{w}(2\pi \sqrt{m}) \sim \Big( \frac{1}{4\sqrt{m}} \Big)^{\frac{1}{2}}e^{-2\pi \sqrt{m}},
$$
so that the $-2\pi\sqrt{m}$ swamps the $\pi\sqrt{3m}$, the
$\pi\sqrt{2m}$, and the $\pi\sqrt{m}$, and the absolute convergence of,
$$
\sum_{m=1}^{\infty} b_r(m) K_{\frac{r s+ 2 - r}{2}}(2\pi \sqrt{m}),
$$
holds for $r =1, 2, 3$, which justifies the interchange of sum and integral.

Had we applied the inversion,
\begin{equation*}
  y\longrightarrow \frac{1}{y}
\end{equation*}
to $\lambda$ alone, the interchange of limits could not be justified. The coefficients, $b_0(m)$, that appear as coefficients in,
$$
\lambda(\tau) = \sum_{m=1}^{\infty}b_0(m) q^m,
$$
(See Simons paper \cite{Simons1952}, (eq 4)),  satisfy
$$
b_0(m) \sim \frac{\pi}{8\sqrt{m}}I_1(2\pi\sqrt{m}).
$$
As 
$$
I_{\nu}(x) \sim \frac{1}{\sqrt{2\pi x}}e^x \qquad x\longrightarrow \infty,
$$
then,
$$
\sum_{m=1}^{\infty} b_0(m)K_w(2\pi \sqrt{m})
$$
does not converge absolutely.

\section{$r=1$}

From the formulas above, with $\frac{r s + 2 - r}{2} = \frac{s + 1}{2}$  for $r = 1$,
\begin{align*}
\Psi_1(s) &= \int_0^{\infty} y^{\frac{s}{2}}\dot F_1 dy \\
 &= \pi \int_0^{\infty} y^{\frac{s}{2}}\sum_{n=1}^{\infty} a_1(n)e^{-\pi n y}\sum_{m=1}^{\infty} b_1(m)e^{-\pi\frac{m}{y}} dy \\
&= \pi \int_0^{\infty}y^{\frac{s-1}{2}} \sum a_1(n)\sum b_1(m)e^{-\pi (n y +  \frac{m}{y})} dy \\
&= \pi  \sum a_1(n)\sum b_1(m) \Big(\frac{m}{n} \Big)^{\frac{s+1}{4}}\int_0^{\infty}y^{\frac{s+1}{2}-1}e^{-\pi \sqrt{mn}(y +  \frac{1}{y})} dy \\
&= 2\pi  \sum a_1(n)\sum b_1(m) \Big(\frac{m}{n} \Big)^{\frac{s+1}{4}}K_{\frac{s+1}{2}}(2\pi \sqrt{mn}).
\end{align*}
The inclusion of the weight one half form, $\theta$, as a factor of $\lambda$, reduces the growth of the $b_1(m)$ significantly in comparison with the growth of $b_0(m)$, with the result that the sum,
$$
\sum_{m=1}^{\infty} b_1(m) K_{\frac{s+1}{2}}(2\pi \sqrt{m}), 
$$
is absolutely convergent. Having obtained a Bessel K expansion of the zeta function, the coefficients, $b_1(m)$, are then of interest and we followed the method of Simons paper to obtain the partition expression of these coefficients. Here we encountered major and minor differences in the $r=1$ calculations from Simon's $r = 0$ calculation (and from the classical partition expansion).

In contrast with the $r=0$ case the sum for $r=1$ in the next section
does not converge absolutely. In fact we did not prove that it
converges. Numerical tests indicate that it does. We evaluated the
arithmetic components of the terms in elementary terms; i. e., sines
and hyperbolic sines.
The case of $r=2$ is like $r=1$ though the
arithmetic components were not all elementary.
 
\begin{conjecture}\label{LambdaCoefficients}
 The coefficients of $\theta \lambda$ are given by,
\begin{equation*}
  b_1(m) = \frac{2\pi }{8\sqrt{2}}\left(\frac{3}{m}\right)^{\frac{1}{4}}  \sum_{k\equiv 2 (mod 4)}  \frac{A_k(m)}{k} \ I_{\frac{1}{2}}\left(\frac{2\pi\sqrt{3m}}{k}\right)
  + \delta_1(m)
\end{equation*}
 where,
\begin{align*}
  A_k(m) & = -i  \sum_{(h, k) = 1} \ \kappa(\gamma^{-1})  e^{-2\pi i (\frac{4mh + 3h'}{4k})},  \\
  \delta_1(m) &=\left\{\begin{array}{cl}
      1 & \mbox{if m is a square} \\
      0 & \mbox{otherwise}
      \end{array}
      \right.
\end{align*}
 and where, $\kappa(\gamma^{-1})$, is a theta multiplier. 
\end{conjecture}

This conjecture was obtained using a method similar to that used
in~\cite{Simons1952}. 

\begin{theorem}\label{AkAsSalie01}
Let $k = 2 l$, $l$ odd. The $A_k(m)$ may be factored as,
\begin{equation}
A_k(m) = i^{\frac{(l+3)(l-1)}{4}} (-1)^{m-1} S(m, -3 B_l^4, l)  \tag{$B_l = \frac{l+1}{2}$}
\end{equation}
  where $S$ is a Salie sum defined as,
$$
S(a, b, l) = \sum_{(h, l) = 1} \Big(\frac{h}{l}\Big) e^{\frac{2\pi i}{l}( a h + b \bar h)}. 
$$
\end{theorem}
\begin{proof}
Using Rademacher's explicit transformation
formula~\cite{Rademacher-AnalyticNumber} for the $\theta_3$-multiplier, it
is a straight-forward calculation to show that the formula holds.
\end{proof}

 The
 Bessel function, $I_{\frac{1}{2}}$, is a hyperbolic sine.
This fact combined with conjecture~\ref{LambdaCoefficients}
and theorem~\ref{AkAsSalie01}, gives a final form for $b_1(m)$:
\begin{align*}
  \frac{\pi}{4\sqrt{2}}&
  \begin{aligned}[t]
    \left(\frac{3}{m}\right)^{\frac{1}{4}}&(-1)^{m-1} \\
    &\sum_{l \ge 1 \ \mbox {odd}}i^{\frac{(l+3)(l-1)}{4}}\frac{1}{2 l} S(m, -3 B_l^4, l) I_{\frac{1}{2}}\left(\frac{\pi \sqrt{3m}}{l}\right) \\
    &+\delta_1(m)
  \end{aligned}
  \\
  &= \begin{aligned}[t]
    \frac{1}{8\sqrt{m}} &(-1)^{m-1}\\
    &\sum_{j=0}^{\infty}\frac{i^{j(j+2)}}{\sqrt{2j+1}} S(m, -3 (j+1)^4, 2j + 1) sinh \left(\frac{\pi \sqrt{3m}}{2j+1}\right) \\
    & +\delta_1(m) .
    \end{aligned}
\end{align*}

The bulk of the remaining calculations are the determination of the Salie sums as trigonometric expressions.

\section{$S(a,b,l)$ Evaluation}

When the Salie sums are not zero they are simple trigonometric
 functions after a $\sqrt{l}$ is extracted.
 We saw that the exponential sum,

$$
S(a, b, l) = \sum_{(h, l) = 1}\Big(\frac{h}{l}\Big) e^{\frac{2\pi i}{l}(a h + b \bar h)},
$$

 evaluates the $A_k(m), k = 2 l,$ as,
\begin{align}
  A_k(m)
  &= i^{\frac{(l+3)(l-1)}{2}} (-1)^{m-1} \ S(m, -3 B^4, l).  \notag
\end{align} 
 We must consider all choices of $m$ and $l$, in which $m$ and $l$ might or might not share factors of $3$ and might or might not share an odd factor other than $3$. All possibilities can occur in the $A_k(m)$. Thus,


\begin{eqnarray*}
m &=& 3^e \times l_1 \times a  \qquad (a, 3l) = 1 \\
l &=& 3^f \times l'_1\times  c  \qquad (c, 3m) = 1\\
(l_1l'_1, 3) &=& 1,  
\end{eqnarray*}
where $l_1$ and $l'_1$ are the factors of $m$ and $l$ that share primes other than $3$. 
The combinations of $e$ and $f$ that modify the values of $S$ do so in more complicated arrangements for $p = 3$ than for the other primes, due to the, $3$, in $S(m, -3B^4, l)$. As we saw, this $3$ comes from the first power of $\theta$ in the formula for the coefficients of $\theta \lambda$. $3$ is no longer special for any higher power of $\theta$. 

\section{Some General Results}

We observe the following general results:
\begin{align}
&1.\ S(a, b, l) = S(b, a, l) \notag\\
&2.\ S(c a, b, l) = \left(\frac{c}{l}\right) S(a, cb, l) \tag{$(c, l) = 1$} \\
&3.\ S(a, b, l) = S(a + m l, b+ n l, l) \notag \\
&4.\ S(a, b, l) = S(a \bar l_1, b \bar l_1, l_2) S(a \bar l_2, b \bar l_2, l_1), \tag{$ l = l_1 l_2, (l_1, l_2) = 1 $}
\end{align}

 It follows from 1 and 2 that if $ a = b = g$, then

$$
S(c g, g, l) =0 \qquad \mbox {if} \ \Big(\frac{c}{l}\Big) = -1.
$$

 If $ c = g_1 \bar g, (g_1 g, l) = 1$, then

$$
S(g_1, g, l) =0 \qquad \mbox{if} \   \Big(\frac{g_1 g}{l}\Big)= -1.
$$

\section{A useful theorem}

A result of~\cite{IwaniecAutomorphic}
evaluates many cases: 
\begin{theorem}\label{IwaniecResult} Let $(l, 2 b) = 1$. Then,
$$
S(a, b, l) = \Big(\frac{b}{l}\Big) \epsilon_l \sqrt{l}\sum_{y^2\equiv ab \ (l)}e^{\frac{4\pi y i}{l}}
$$ 
\end{theorem}

\begin{corollary}
 Let $p\ge 3$ be a prime. If $(a, p) = 1$, $e\ge 1$, and $f\ge 2$, then,
$$
S(a, p^e b, p^f) = 0 
$$ 
\end{corollary}

\section{$p > 3$}

For a prime, $p > 3$ dividing $l_1$ and $l_1'$, a sum of the form,
\begin{equation*}
S(p^e a, 3 b, p^f), (b, p) = 1,
\end{equation*}
will be a factor of
\begin{equation*}
S(m, -3 B^4, l)
= S(3^e l_1 a, -3B^4, 3^f \times l'_1\times c).
\end{equation*}
The few cases are
contained in the following result.

\begin{theorem}\label{AkAsSalie02}
 \quad  Let $p > 3$ be a prime with $(p, ab) = 1$ and let $S = S(p^e a , 3 b, p^f)$. Then,

\begin{align*}
                e &\ge 1, \ f\ge 2,  \quad S = 0 \\
                e &\ge 1, \ f = 1, \  \quad S = \Big(\frac{3b}{p}\Big) \epsilon_p \sqrt{p} \\
                e &= 0, \ f\ge 1,  \quad S = \Big(\frac{3b}{p}\Big)\epsilon_{p^f} \sqrt{p^f} \sum_{y^2\equiv = 3ab (p^f)}e^{\frac{2\pi i y}{p^f}}  \\                        
\end{align*}
\end{theorem}
\begin{proof}
All of these cases are contained in the Iwaniec formula. We separated the ones that vanished, and the simple Gauss sum, from the last case. 
\end{proof}

\section{$p = 3$}

 The combination of powers of $3$ dividing $m$ and $l$ can be grouped according to $e = 0, 1, 2$, where we have let,

\begin{eqnarray*}
m &=& 3^e \times l_1 \times a  \qquad (a, 3l) = 1 \\
l &=& 3^f \times l'_1\times  c  \qquad (c, 3m) = 1 \qquad  (l_1l'_1, 3) = 1.
\end{eqnarray*}

 The (twisted) multiplicativity result, 4, of the note, shows that a sum of the type,  $S = S(3^e a, 3 b, 3^f)$, is a factor of $S(3^e l_1 a, -3B^4, 3^f \times l'_1\times  c)$. The situation regarding this factor is given in the following result. 

\begin{theorem}\label{S3powers}
Let $S = S(3^e a , 3 b, 3^f)$. Then, 
\begin{align*}    
e &\ge 2,  f = 2,   S= -3 \\
e &\ge 2,  f\ge 3,  S=0     \\ 
e &\ge 1,  f = 1,  S = 0  \\ 
e &= 1,  f = 2,  S = 6,   \left(\frac{ab}{3}\right) = -1  \\
e &= 1,  f = 2,  S = -3,  \left(\frac{ab}{3}\right) = 1    \\
e &= 1,  f \ge 3,  S = 0 \qquad   
 \mbox{$ab$ is not a square mod $3^f$} \\
 e &
 \begin{aligned}[t]
   = 1,  f = 2n, n\ge 2,   S &= -\left(\frac{g}{3}\right) 2\times 3^n \sqrt{3} \  Sin\left(\frac{4 \pi g}{3^{2n-1}}\right), \\
   g^2&\equiv ab (3^{2n})
   \end{aligned} \\
 e & \begin{aligned}[t]
   = 1,   f = 2n-1, n\ge 2,  6S &=   2 i\left(\frac{g}{3}\right) \times 3^{n+1}  \  Sin\left(\frac{4 \pi g}{3^{2n-2}}\right), \\
    g^2&\equiv ab (3^{2n-1})
   \end{aligned}\\
e &= 0,  f\ge 2,  S = 0 \\
  e&=0,  f=1, S=\left(\frac{a}{3}\right)\sqrt{3}i
\end{align*}
\end{theorem}
\begin{proof}
The two $e = 0$ cases are direct results of Theorem~\ref{IwaniecResult}. The results for $e \ge 1$ follow from a formula that generalizes Theorem~\ref{IwaniecResult} in an obvious way. We worked out each of the cases that can occur and presented them in their simplest form. 
\end{proof}

\section{$r = 2$}

With $r = 2$, then $\frac{r s + 2 - r}{2} = s$, and

\begin{eqnarray*}
\zeta^{*}(s, 2)   &=& \frac{1}{4} \int_0^{\infty} y^{s-1}(\theta^2 -1) dy  \\
&=&  \pi^{-s} \Gamma(s) \zeta(s, 2) \\
&=&  \pi^{-s} \Gamma(s) \zeta(s) L(s, \chi), 
\end{eqnarray*}

\begin{theorem}\label{BesselExpansionrIs2}
$$
4 \zeta^{*}(s, 2) =   -\left(\frac{\Psi_2(s)}{s} + \frac{\Psi_2(1-s)}{1-s}\right),
$$
$$
\Psi_2(s) = 2\pi \sum_{j=1}^{\infty} c_j(2, s)K_s(2\pi \sqrt{j}) 
$$
$$
c_j(2, s) = \sum_{d|j}a_2(d) b_2\left(\frac{j}{d}\right)\left(\frac{j}{d^2}\right)^{\frac{s}{2}}
$$
\end{theorem}

\begin{conjecture}\label{b2mEvaluation}
\begin{equation*}
b_2(m) =  \frac{\pi}{2}\sum_{k\equiv 2 (mod 4)}  \frac{A_k(m)}{k} \ I_0\left(\frac{2\pi\sqrt{2m}}{k}\right) + 2\delta_2(m) 
\end{equation*}
where
$$
\frac{\theta^2 -1}{4} = \sum_{n=1}^{\infty} \delta_2(m) q^m,
$$
 and
 $$
A_k(m) = -  \sum_{(h, k) = 1} \ \kappa^2(\gamma^{-1})e^{-2\pi i (\frac{h'}{2k} +  \frac{mh}{k})}  
$$
 is the arithmetic factor with the same multiplier, $\kappa(\gamma^{-1})$, as before, but squared.
\end{conjecture}

\begin{theorem}\label{ris2Zeta}
\begin{equation}
A_k(m) = (-1)^{(l-1)/2}(-1)^{m-1} T(-2m A^2, A^2, l), \tag{$k=2l$, $A = \frac{l+1}{2}$}
\end{equation}
where $T$ is the Kloosterman sum:
$$
T(a, b, l) = \sum_{(h,k)=1}e^{\frac{2\pi i}{l}(a h + b \bar h)}.
$$
\end{theorem}

Combining conjecture~\ref{b2mEvaluation} and theorem~\ref{ris2Zeta} above provides a final evaluation for the $b_2(m)$:
$$
b_2(m) = \frac{\pi}{4}(-1)^{m-1}\sum_{j=0}^{\infty}\frac{T(-2m(j+1)^2, (j+1)^2, 2j+1)}{2j+1} + 2 \delta_2(m).
$$

\section{$r = 3$}

The function,
$$
\zeta^{*}(s, 3) = \frac{1}{6} \int_0^{\infty} y^{s-1}(\theta^3 -1) dy  
$$
has the Bessel K series expansion,

\begin{theorem}
$$
6 \zeta^{*}(s, 3) =   -\Big(\frac{\Psi_3(s)}{s} + \frac{\Psi_3(1-s)}{1-s}\Big),
$$
$$
\Psi_3(s) = 2\pi \sum_{j=1}^{\infty} c_j(3, s)K_{\frac{3s-1}{2}}(2\pi \sqrt{j}) 
$$
$$
c_j(3, s) = \sum_{d|j}a_3(d) b_3(\frac{j}{d})\Big(\frac{j}{d^2}\Big)^{\frac{3s-1}{4}}
$$
\end{theorem}

Following the same steps in the calculations for $b_1$ and $b_2$ we obtained the expression,
\begin{equation}\label{b3mEquation}
b_3(m) =  \sqrt{2}\pi \ m^{\frac{1}{4}} \ \sum_{k\equiv 2 (mod 4)}  \frac{A_k(m)}{k} \ I_{-\frac{1}{2}}\Big(\frac{2\pi\sqrt{m}}{k}\Big) + 3\delta_3(m).
\end{equation}
where 
\begin{equation*}
\frac{\theta^3 -1}{6} = \sum_{n=1}^{\infty} \delta_3(m) q^m,
\end{equation*}
 and
$$
A_k(m) = - i \sum_{(h, k) = 1} \ \kappa^3(\gamma^{-1})e^{-2\pi i (\frac{h'}{4k} +  \frac{mh}{k})} 
$$
We do not believe that the sum in equation~\ref{b3mEquation} converges.

\begin{theorem}\label{ris3AkmCalculation}
  \begin{equation}
 A_k(m) = (-1)^{m-1} i^{\frac{l-1}{2}} S(-4m A^3, A^3, l) \tag{$A =  \frac{l+1}{2}, k = 2l$}
  \end{equation}
  where
  $$
S(a, b, l) = \sum_{(h, l) =1} \left(\frac{h}{l}\right)e^{\frac{2\pi i}{l}(a h + b \bar h)}.
$$
\end{theorem}

\appendix

\section{}

Assume that for arithmetic sequences $[a(n)]_{n=1}^\infty$ and $[b(n)]_{n=1}^\infty$, the function
\begin{equation*}
\Psi(s) = \sum_{j=1}^{\infty} c_j(s)K_s(2\pi \sqrt{j}),
\end{equation*}
is entire, where, as defined earlier,
\begin{equation*}
c_j(s) = \sum_{d|j} a(j)b(\frac{d}{j})\left(\frac{j}{d^2} \right)^{\frac{s}{2}} \qquad j\ge 1
\end{equation*}
For the special sequences that appear in Theorem~\ref{ris4}, numerical
tests for the zeros of $\Psi$ gave what was expected: zeros on four
vertical lines, $\sigma=-\frac{1}{2}, \frac{1}{2},0,1$. For
$a(n)=b(n)=1$, the zeros of $\Psi$ that we found were pure imaginary to
within machine precision. Other pairs of multiplicative sequences gave
zeros that were sporadically distributed. A rationale for which
sequences produced highly organized zeros for the associated $\Psi$ is
an attactive puzzle. We considered briefly an operator that would
encompass the second order equations satisfied by the $K_s(x_j)$, but dropped it. We take up a curious property of Hecke operators in Appendix~\ref{partitionAppendix}.

\section{Other partitions of unity}\label{partitionAppendix}

We looked at other partitions of unity that may be used to provide Bessel
expansions for the Riemann Zeta function. For a modular function of
weight zero to be a partition of unity it must satisfy the functional
equation
\begin{equation*}
  f(\tau)+f\left(-\dfrac{1}{\tau}\right)=1.
\end{equation*}
If $f$ is such a partition of unity then this leads to a partition of  the Riemann
$\xi$ of the form
\begin{equation}\label{psifunctional}
  2\xi(s) = (1-s)\Psi_f(s) + s\Psi_f(1-s)
\end{equation}
where
\begin{equation}\label{psidefinition}
  \Psi_f(s) = \int_0^\infty x^{s/2}
  \left(\dfrac{d}{dx}\theta_3(ix)f(ix)\right)dx.
\end{equation}
Certain partitions of unity accumulated zeros of the associated $\Psi(s)$ close to the $\frac{1}{2}$ line.  As a first example, consider the zeros of 
$\Psi_{1-\lambda}$:
\begin{equation*}
\begin{split}
-6.85993 + 18.5302 i \\
8.13082 + 18.4858 i \\
 -7.17924 + 28.5267 i \\
8.33654 + 28.5203 i \\
-8.228 + 36.2544 i \\
9.40533 + 36.2529 i \\
 -8.20765 + 42.7431 i \\
9.33192 + 42.7147 i \\
 -8.69688 + 49.6138 i \\
 9.82173 + 49.6313 i \\
-9.19901 + 55.0839 i \\
 10.3377 + 55.0643 i.
\end{split}
\end{equation*}

An infinite class of partitions of unity arise from weight one-half Hecke operators acting on weight one-half modular forms.
For an odd
prime, $p$, the Hecke operator acting on a weight one-half modular form $g$, is expressed as
\begin{equation*}
   {\mathcal H}_p(g)(\tau) = \begin{aligned}[t] g(p^2\tau) 
             & + \dfrac{1}{\sqrt{p}\epsilon_p}\sum_{m=1}^{p-1}\left(\dfrac{m}{p}\right)g\left(\tau+\dfrac{2m}{p}\right) \\
     & + \dfrac{1}{p} \sum_{m=0}^{p^2-1}g\left(\dfrac{\tau+2m}{p^2}\right),
     \end{aligned}
\end{equation*}
where
\begin{equation*}
             \epsilon_p = \left\{\begin{array}{cl} 1 & p\equiv 1(4) \\
                                         i & p\equiv 3(4)
                      \end{array}\right.
\end{equation*}
and 
\begin{equation*}
  \left(\dfrac{m}{p}\right)
\end{equation*}
is a Jacobi symbol. If $\phi$ is a partition of unity then
\begin{equation*}
  \dfrac{1}{p+1}\dfrac{1}{\theta_3}\mathcal H_p(\theta_3\phi)
\end{equation*}
is also a partition of unity.

For example, if
\begin{equation*}
  f_1(\tau) = \theta_3(\tau)(1-\lambda(\tau)),
\end{equation*}
then
\begin{align*}
  \dfrac{1}{4}\mathcal H_3&(f_1)(\tau) = \\
  &= \dfrac{1}{4}
  \begin{aligned}[t]
    \Big(f_1(9\tau)
  &-\dfrac{i}{\sqrt{3}}f_1\left(\tau+\dfrac{2}{3}\right)
  +\dfrac{i}{\sqrt{3}}f_1\left(\tau+\dfrac{4}{3}\right) \\
  &+\dfrac{1}{3}\sum f_1\left(\dfrac{\tau+2k}{9}\right)\Big)
  \end{aligned}\\
  & = \theta_3(\tau)(1 - \lambda(\tau))
  \begin{aligned}[t](1 &- 12000 \lambda(\tau) + 441792 \lambda(\tau)^2 \\
    & - 3350528 \lambda(\tau)^3 
    + 9224192 \lambda(\tau)^4 \\ &- 10485760 \lambda(\tau)^5 + 
    4194304 \lambda(\tau)^6).
    \end{aligned}
\end{align*}
is a partition of unity attached to the prime, 3.

\begin{figure}
  \includegraphics[width=4.5in]{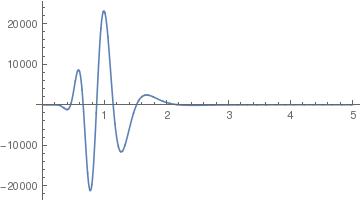}
  \caption{Plot of the integrand of the $\Psi$ function obtained from 
    $\mathcal H_3$ of $\theta_3(1-\lambda)$}\label{PsiIntegrand03}
\end{figure}
\begin{figure}
  \includegraphics[width=4.5in]{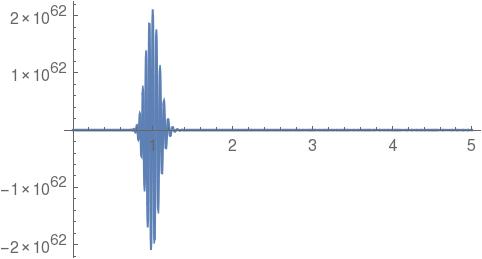}
  \caption{Plot of the integrand of the $\Psi$ function obtained from 
    $\mathcal H_{11}$ of $\theta_3(1-\lambda)$}\label{PsiIntegrand11}
\end{figure}

Regarding numeric evaluations of the $\Psi$ functions for these partitions of unity, we used
Mathematica's numeric integration technology to perform the
integrations for equation~\ref{psidefinition}.  We were a bit worried
about the accuracy of such calculations because as the prime, $p$,
gets larger, the behavior of the integrand becomes extremely
oscillatory.
For example, for $p=3$ (see
figure~\ref{PsiIntegrand03}), the integrand
goes between -20,000 to +20,000 but the curve is smooth. But as
the prime $p$ gets larger the oscillations of the integrand becomes more
extreme.  For example, 
figure~\ref{PsiIntegrand11}
shows a plot of the integrand when $p=11$.

Concerned that mathematica could not accurately calculate the value of $\Psi$ when the integrand was behaving in such an extreme manner, we tested our
numerical values for $\Psi$ in the equation~\ref{psifunctional} on page~\pageref{psifunctional} and
the results were very encouraging. For example, when $s=.5+i$ numeric calculations showed
\begin{align*}
  \xi(s)      & = -0.485757 \\
  \Psi_{11}(s) & = 4.738132-2.85482i\\
  2\xi(s) -( (1-s)\Psi_{11}(s) + s\Psi_{11}(1-s)) & = 0 
\end{align*}
These results continued to be encouraging for other larger values of
$s$. For example at $s=50+50i$ we obtained
\begin{align*}
  \xi(s)      & = -8.08211*10^{9}+1.715833*10^9 i\\
  \Psi_{11}(s) & = -5.135*10^{38}-3.14968*10^{38} i \\
  2\xi(s) -&( (1-s)\Psi_{11}(s) + s\Psi_{11}(1-s)) \\
  & = -5.63493*10^{-19}+3.19468*10^{-19}i
\end{align*}
This seemed like pretty convincing evidence because it did not
seem likely that Mathematica was using integration by parts in such a
manner as to obtain erroneous results that were consistent with~\ref{psifunctional}.

It may be worth noting that in order to avoid errors in Mathematica we
had to use some rather extreme values on the numeric integration
precision.  For example, we set Mathematica up to perform integration
with a working precision of 200 places of accuracy and a precision goal
of 30 places of accuracy. We noticed that with these settings
Mathematica seemed to return results that were significantly more
accurate than requested.

Regarding the zeros clustering phenomenon around the
one-half line, the behavior was not particularly striking for the $\Psi$-function associated with $\mathcal H_3$ versus the previous example, but it was apparent. Thus, for zeros in the square
between -50 and $50+40i$:
\begin{equation*}
  \begin{array}{cc}
     0.235237 + 0.77205 i \\
     0.516168 + 3.6727 i \\
     0.498618 + 6.13523 i \\
-21.1074 + 19.7716 i & 22.1074 + 19.7716i  \\
-23.4598 + 29.768 i  & 24.4599 + 29.768 i  \\
-25.2761 + 37.4269 i &  26.2762 + 37.4269i  
\end{array}
\end{equation*}
The  $\Psi$-function associated with $\mathcal H_5$ of $1-\lambda$ gave the following zeros in the square from -50 to
$50+50i$:
\begin{align*}
0.233242 + 0.437961 i\\
 0.519924 + 2.43977 i\\
 0.49621 +  4.48092 i\\
 0.500645 + 6.21101 i\\
 0.499978 + 9.29704 i\\
 0.500002 + 11.2699 i\\
 0.5 + 13.5882 i\\
 0.5 + 17.5795 i\\
 -34.6505 + 21.523 i \\ 35.6505 +  21.523 i\\
 0.5 + 22.4764 i \\
 -38.8576 + 31.7487 i \\  39.8576 +  31.7487 i\\
 -41.6848 +  39.6315 i \\ 42.6848 + 39.6315 i\\
 -43.8901 + 46.6857 i \\  44.8901 + 46.6857 i
\end{align*}
The occurances of a real part of $.5$ does not mean that the zero is 
exactly on the one-half line; it is just means that such a zero is
within machine precision of the one-half line.
The zeros of $\Psi$-function associated with $\mathcal H_7$ of
$1-\lambda$ show a greater clustering behavior.  The results for zeros in the box between -50 and $50 +50i$:
\begin{equation*}
\begin{split}
  0.234502 + 0.298452 i\\
  0.519389 + 2.01942 i\\
  0.495116 +  3.64979 i\\
  0.501157 + 5.24145 i \\
  0.49982 + 6.98016 i\\
  0.500017 +  9.04737 i\\
  0.499999 + 11.023 i\\
  0.5 + 12.9791 i\\
  0.5 + 15.0838 i\\
-1.01888 + 18.9618 i \\  2.01888 + 18.9618 i\\ 
  0.5 +  22.496 i\\
-47.6872 + 23.3824 i \\ 48.6872 + 23.3824 i\\ 
0.5 +  25.9916 i\\
0.5 + 29.3926 i\\
0.5 + 33.1953 i\\
0.5 + 37.4991 i\\
0.5 +  42.5267 i\\
0.5 + 48.8834 i
\end{split}
\end{equation*}
Finally, the zeros of the $\Psi$-function associated with $\mathcal
H_{11}$ of $1-\lambda$ in the box from $-50$ to $50+50i$:
\begin{align*}
0.236946 + 0.139035 i \\
 0.517157 + 1.64704 i \\
 0.494551 +  2.94259 i \\
 0.501784 + 4.27374 i \\
 0.499487 + 5.50646 i \\
 0.500101 +  7.19582 i \\
 0.499972 + 8.26041 i  \\
 0.500003 + 10.6066 i \\
 0.499998 +  11.0806 i \\
 0.5 + 13.5367 i \\
 0.5 + 15.0404 i \\
 0.5 + 16.5906 i \\
 0.5 +  18.7044 i \\
 -0.00762622 + 21.1219 i \\
 1.00763 + 21.1219 i \\
 -0.423696 + 26.1218 i \\
 1.4237 +  26.1218 i \\
 0.5 + 26.6352 i \\
 0.5 +  31.5323 i \\
 0.5 + 32.6608 i \\
 -1.6787 + 35.8675 i \\
 2.6787 +  35.8675 i \\
 0.119661 + 39.9134 i \\
 0.880339 + 39.9134 i \\
 0.5 + 
 44.4793 i \\
 0.5 + 48.8203 i
\end{align*}

The exceptional zeros suggested we try for hybrids of Hecke operators arising from collections of primes, and we found a sum, which, formally, was an eigenvector
for the Hecke operators for all odd primes, but we could not prove that the series converged.

There are partitions of unity that come from pieces of the Hecke
operator. For example,
\begin{equation*}
  f(\tau)=1-\lambda\left(\dfrac{\tau+a}{p}\right)
\end{equation*}
is a partition of unity if $a$ is even, $p$ is an odd prime and
\begin{equation*}
  a^2=-1 \mbox{ mod } p.
\end{equation*}
This example can be generalized. Such functions are algebraic functions of $\lambda$ although their coefficients grow too quickly to produce a Bessel K expansion of the $\zeta$
function. Another class of  algebraic functions of $\lambda$ that are also partitions
of unity may be obtained using the lemma below. In some cases, these have the advantage that their coefficients go to infinity slowly enough that they may be used to provide Bessel K expansions for the zeta function.

\begin{lemma}\label{AlgebraicPartition}
  Suppose that $R(x)$ and $S(x)$ are polynomials of odd and even
  degrees respectively with real
  coefficients such that 
  \begin{itemize}
  \item $R(1-x)+R(x)=S(x)$
  \item The polynomial $R(x)S'(x)-R'(x)S(x)$ is zero only when $x=0$
    or $x=1$
  \item $S(x)$ has no real zeros when $0< x < 1$
  \item $S(1)=R(1)\neq 0.$
  \end{itemize}
  Then there is a modular function, $\phi$, over a finite index subgroup of the
  $\Lambda$ group which satisfies the following conditions:
  \begin{align*}
    R(\phi(\tau))-\lambda(\tau)S(\phi(\tau)) &= 0 \\
    \phi(\tau)+\phi(-1/\tau) &= 1 \\
    \phi(0) &= 1 \\
    \phi(i\infty) &= 0 \\
    \phi(i) &= 1/2.
  \end{align*}
\end{lemma}
\begin{proof}
  For any fixed $y$, the equation,
  \begin{equation*}
    R(x)-y S(x) = 0,
  \end{equation*}
  will only have a double root in $x$ when
  \begin{equation*}
    R'(x)-y S'(x) = 0.
  \end{equation*}
  This means that,
  \begin{equation*}
    R'(x)S(x)-R(x)S'(x) = 0,
  \end{equation*}
  so $x$ must be either zero or one. It then follows that $y$ is
  either zero or one.

  Now if we  consider $x$ and $y$ to be functions of $\tau$ with, say, 
  $x=\phi(\tau)$ and 
  $y=\lambda(\tau)$, then there are no double roots in the $x$ variable
  of 
  \begin{equation*}
    R(x)-y S(x) = 0,
  \end{equation*}
  when $\tau$ is in the upper half plane. In addition, since the
  degree of $R(x)$ is at least one more than the degree of $S(x)$, $x$
  cannot go to infinity for $\tau$ in the upper half plane.

  Thus, for any given $\tau$ in the upper half plane, any solution,  $\phi(\tau)$, of the equations
  \begin{equation}\label{phiEquation}
    R(\phi(\tau)) - \lambda(\tau) S(\phi(\tau))=0
  \end{equation}
  can be extended to a single-valued root for all $\tau$ in the upper
  half plane.

  Now we will consider the behavior of
  \begin{equation*}
    \dfrac{R(x)}{S(x)}
  \end{equation*}
  as $x$ goes from zero to one in order to fix attention on a specific
  root of
  \begin{equation*}
    R(\phi(\tau)) - \lambda(\tau) S(\phi(\tau))=0.
  \end{equation*}
  We have
  \begin{equation*}
    \dfrac{R(1)}{S(1)}=1
  \end{equation*}
  and
  \begin{equation*}
    \dfrac{R(0)}{S(0)}=0.
  \end{equation*}
  In addition, for $x$ in the open interval from 0 to 1,
  \begin{equation*}
    \dfrac{d}{dx}\left(\dfrac{R(x)}{S(x)}\right)
    = \dfrac{R'(x)S(x)-R(x)S'(x)}{S(x)^2} \neq 0.
  \end{equation*}
  This means that 
  \begin{equation*}
    \dfrac{R(x)}{S(x)}
  \end{equation*}
  is monotonically increasing for $x$ in the open interval from zero
  to one.  In addition,
  \begin{equation*}
    \dfrac{R(x)}{S(x)}+\dfrac{R(1-x)}{S(1-x)}=1.
  \end{equation*}
  
  So
  \begin{equation*}
    \dfrac{R(1/2)}{S(1/2)}=1/2.
  \end{equation*}
  We thus can choose the root of 
  \begin{equation*}
    R(\phi(\tau)) - \lambda(\tau) S(\phi(\tau))=0,
  \end{equation*}
  with
  \begin{equation}\label{PhiOfI}
    \phi(i) = 1/2.
  \end{equation}

  Now we use~\ref{PhiOfI} to show that the solution, $\phi$,
  of equation~\ref{phiEquation} must be a partition of unity.  The
  conditions of lemma~\ref{AlgebraicPartition} have been crafted so
  that if
  \begin{equation*}
    R(x_0)-y_0 S(x_0) = 0,
  \end{equation*}
  then
  \begin{equation*}
    R(1-x_0)-(1-y_0) S(1-x_0) = 0.
  \end{equation*}
  This means that
  \begin{equation*}
    R(1-\phi(-1/\tau)) - (1-\lambda(-1/\tau)) S(1-\phi(-1/\tau))=0,
  \end{equation*}
  or
  \begin{equation*}
    R(1-\phi(-1/\tau)) - \lambda(\tau) S(1-\phi(-1/\tau))=0.
  \end{equation*}
  Thus, by looking at what happens in a neighborhood of $\tau=i$ we get
  \begin{equation*}
    \phi(\tau)=1-\phi(-1/\tau).
  \end{equation*}
  This completes the proof.
\end{proof}

The algebraic equations
\begin{align*}
    2\phi(\tau)^3-3\phi(\tau)^2+\lambda(\tau) &= 0 \\
    \phi(\tau)^5-\lambda(\tau)(5\phi(\tau)^4-10\phi(\tau)^3+10\phi(\tau)^2-5\phi(\tau)+1)    &= 0 \\
    \begin{aligned}[b]
      2 \phi(\tau)^7&-7\phi(\tau)^5 \\
      &+\lambda(\tau)(35\phi(\tau)^4-70\phi(\tau)^3+63\phi(\tau)^2-28\phi(\tau)+5)
    \end{aligned}
    &=0 \\
    \begin{aligned}[b]
      2\phi(\tau)^{11}&-11\phi(\tau)^{10} \\
    &\begin{aligned}[b]
      +\lambda(\tau)(165\phi(\tau)^8&-600\phi(\tau)^7+1386\phi(\tau)^6 \\
      &-1848\phi(\tau)^5+1650\phi(\tau)^4 \\
      &-990\phi(\tau)^3+385\phi(\tau)^2-88\phi(\tau)+9)
     \end{aligned}
    \end{aligned} &= 0
\end{align*}
satisfy the conditions of lemma~\ref{AlgebraicPartition}. At this
time, we don't know if there are other similar equations that can be
derived using lemma~\ref{AlgebraicPartition} or if a more
sophisticated lemma would produce more roots.

We will now  provide some expanded analysis of the first of these equations:
\begin{equation}\label{CubicEquation}
  2\phi(\tau)^3-3\phi(\tau)^2+\lambda(\tau)=0.
\end{equation}
For $\tau$ in the upper half plane, this equation has no double roots
and thus it defines three single valued functions, $\phi_1$, $\phi_2$
and $\phi_3$, on the upper half
plane.

The partition of unity root of equation~\ref{AlgebraicPartition} is the the root, $\phi_1$, where
\begin{equation*}
  \phi_1(i)=\dfrac{1}{2}.
\end{equation*}
It is easy to see (see
figure~\ref{CubicPlot})
that for the $\phi_1$
root we have
\begin{align*}
  \phi_1(0) & = 1 \\
  \phi_1(\infty) & = 0 \\
  \phi_1(\tau)+\phi_1(-1/\tau) &= 1.
\end{align*}
The other two roots can be characterized by their evaluation at
$\tau=i$:
\begin{align*}
  \phi_2(i) &= \dfrac{1-\sqrt{3}}{2} \\
  \phi_3(i) &= \dfrac{1+\sqrt{3}}{2}.
\end{align*}
It is also easy to see (see
    figure~\ref{CubicPlot})
    that
\begin{align*}
  \phi_2(0)&=-\dfrac{1}{2} \\
  \phi_3(0)&=1 \\
  \phi_2(i\infty) &= 0 \\
  \phi_3(i\infty) &= \dfrac{3}{2} \\
  \phi_2(\tau)+\phi_3(-1/\tau) &= 1.
\end{align*}
From this last identity, we could use $\phi_2$ and $\phi_3$ also to
provide partions of unity.

\begin{figure}
  \includegraphics[width=4.5in]{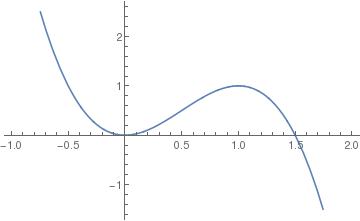}
  \caption{Plot of $3\phi^2-2\phi^3$ for $\phi\in (-1,2)$}\label{CubicPlot}
\end{figure}

The $\lambda$-group ($\Gamma[2]$) permutes $\phi_1$, $\phi_2$ and $\phi_3$. If
\begin{align*}
  \lambda_1 &= \left(\begin{array}{cc} 1 & 2 \\ 0 & 1 \end{array}\right) \\
  \lambda_2 &= \left(\begin{array}{cc} 1 & 0 \\ 2 & 1 \end{array}\right),
\end{align*}
then
\begin{align*}
  \phi_1\circ\lambda_1 &= \phi_2 \\
  \phi_2\circ\lambda_1 &= \phi_1 \\
  \phi_3\circ\lambda_1 &= \phi_3 \\
  \phi_1\circ\lambda_2 &= \phi_3 \\
  \phi_2\circ\lambda_2 &= \phi_2 \\
  \phi_3\circ\lambda_2 &= \phi_1.
\end{align*}

It then follows from combinatorial group theory
that the subgroup of $\Gamma[2]$ that leaves $\phi_1$
invariant is freely generated by the elements
\begin{align*}
  \lambda_1^2         &= \left(\begin{array}{cc} 1 & 4 \\ 0 & 1 \end{array}\right) \\
  \lambda_2^2         &= \left(\begin{array}{cc} 1 & 0 \\ 4 & 1 \end{array}\right) \\
  \lambda_1\lambda_2\lambda_1^{-1} &= \left(\begin{array}{cc} 5 & -8 \\ 2 & -3 \end{array}\right) \\
  \lambda_2\lambda_1\lambda_2^{-1} &= \left(\begin{array}{cc} -3 & 2 \\ -8 & 5 \end{array}\right).
\end{align*}

Simon's method may be applied to $\phi_1$ and
used to
give another BesselK expansion for $\zeta(s)$.
The main term of Simon's formula for the
coefficients of $\phi_1$
gives an order of magnitude of the mth-coefficient of
$\phi_1(q)$ as
\begin{equation*}
   \dfrac{1}{m}\exp{(2\pi\sqrt{m/3})}.
\end{equation*}
This order of magnitude estimate allows the interchange of summation
and integration.

This order of magnitude estimate is consistent with numeric
calculations of the coefficients wherein, for example, the 1000th coefficient of
$\phi_1(q)$ is
\begin{equation*}
  -2.21563*10^{46}
\end{equation*}
and
\begin{equation*}
  \dfrac{1}{1000}exp{\left(2\pi\sqrt{1000/3}\right)}=6.60664*10^{46}.
\end{equation*}

We searched for zeros of $\Psi_{\phi_1}$ inside the box extending to
-20 on the left and +20 on the right and going up from 0 to 40:
\begin{equation*}
\begin{array}{lr}
    -0.196822 + 9.837 i
    & 0.773957 + 12.8053 i \\
    -0.305477 + 19.939 i  &
     1.67883 + 19.3617 i \\
    -1.51671 + 26.1875 i &
     2.14412 + 25.9197 i \\
    -0.729122 +  31.5098 i &
    1.73372 + 31.7993 i\\
   -1.18117 +  36.8074 i &
     2.50435 + 36.6545 i.
\end{array}
\end{equation*}

\bibliography{Modular}{}

\end{document}